\documentclass{amsart}
\usepackage{amssymb,amsthm,amsmath}
\usepackage{bm}
\usepackage{esint}
\usepackage[backend=biber,style=numeric,sorting=nyt,doi=true,isbn=false,url=false]{biblatex}
\usepackage[hidelinks]{hyperref} 
\usepackage[noabbrev]{cleveref}
\usepackage{graphicx}
\usepackage{subfigure}

\newtheorem{theorem}{Theorem} 

\newtheorem{lemma}{Lemma}
\newtheorem{proposition}[lemma]{Proposition}

\theoremstyle{definition}

\newtheorem{remark}[lemma]{Remark}

\numberwithin{equation}{section}

\newcommand*{\de}{\partial}
\newcommand*{\p}{\partial}
\newcommand*{\N}{\mathbb{N}}
\newcommand*{\Z}{\mathbb{Z}}
\newcommand*{\R}{\mathbb{R}}

\renewcommand*{\S}{\mathbb{S}}

\newcommand*{\X}{\mathcal{X}}

\newcommand*{\loc}{{\text{\upshape{loc}}}}
\newcommand*{\lap}{\triangle}

\DeclareMathOperator{\im}{Im}

\DeclareMathOperator{\dist}{dist}

\author{Federico Franceschini}
\address{Institute for Advanced Study, Princeton, USA}
\email{ffederico@ias.edu}
\date{\today}
\author{Ovidiu Savin}
\address{Columbia University, New York, USA}
\email{savin@math.columbia.edu}

\title{Solutions to the Thin Obstacle Problem with non-2D frequency}
\begin{document}
\begin{abstract}
     For all odd positive integers $m$, we construct $\mu$-homogeneous solutions to the thin obstacle problem in $\mathbb{R}^3,$ with $\mu\in(m,m+1)$. For $m$ large, $\mu-m$ converges to $1$, so $\mu\neq m+\tfrac 1 2$.
     
    The restriction to odd values of $m$ is necessary: we show that, for all $n\ge 2$, there are no $\mu$-homogeneous solutions to the thin obstacle problem in $\mathbb{R}^n$ with $\mu \in \bigcup_{k\ge 0}(2k,2k+1)$.
    
    These examples also apply to $2$-valued $C^{1,1/2}$ stationary harmonic functions or $\mathbb{Z}/2\mathbb{Z}$-eigenfunctions of the laplacian on the sphere.     
\end{abstract}
\maketitle
\section{Introduction}
\subsection{The thin obstacle problem}

A solution to the the thin obstacle problem in $B_1 \subset \R^n$ is a continuous function $u$, even with respect to the thin space $\{x_n=0\}$, that satisfies
$$ \lap u = 0 \quad \mbox{in} \quad B_1 \setminus Z(u), \quad \quad Z(u):=\{x_n=0\} \cap \{u=0\},$$
and
\begin{equation}\label{eq:TOP}
u \ge 0\quad\text{and}\quad u_n \le 0\quad\text{on}\quad \{x_n=0\},
\end{equation}
where $u_n$ denotes the one-sided directional derivative in the positive $x_n$ direction.
 
Homogeneous solutions are of particular interest, since they appear in the blow-up analysis at points on the free boundary 
 $$\Gamma(u)=\de_{\{x_n=0\}} \, Z(u).$$ 
The homogeneity $\mu>0$ of a homogeneous solution is called its frequency.

Recall that, if $w$ is an homogeneous function of degree $\mu$ in $\R^n$, then on the unit sphere
$$\lap w =  \, L_\mu w ,\quad \quad \quad  \quad L_\mu w:=\triangle_{\S^{n-1}} w +\mu(\mu+n-2)w.$$

Thus finding $\mu$-homogeneous solutions of \eqref{eq:TOP} is equivalent to find solutions of the following problem on the upper hemisphere:
\begin{equation}\label{eq:TOPsphere}
\begin{cases}
L_\mu u =0 &\text{ in }\{x_n>0\}\cap \S^{n-1},\\ 
u\ge 0&\text{ on }\{x_n=0\}\cap\S^{n-1},\\
u_n\le 0&\text{ on }\{x_n=0\}\cap\S^{n-1},\\
u\cdot u_n=  0&\text{ on }\{x_n=0\}\cap\S^{n-1}.
\end{cases}
\end{equation}

\subsection{Main results}
Our convention for spherical coordinates in 3D is:
\begin{equation*}
    x = r \cos \theta\cos\varphi,\quad y = r\sin\theta\cos\varphi,\quad z = r\sin\varphi,
\end{equation*}
with $\theta\in \R/2\pi\Z$ and $\varphi\in[-\pi/2,\pi/2].$ The thin space is $\{z=0\}=\{\varphi=0\}$.
\begin{theorem}\label{thm:main}
    For each odd integer $m\ge1$ there is a homogeneous solution of the thin obstacle problem $u\colon \R^3\to \R,$ with homogeneity $\bar\mu_m\in (m,m+1)$ and contact set
    \begin{equation*}
        \{u =0 \}\cap\{ z=0\} =\{\dist(\theta,\tfrac{2\pi}{m}\Z )\le \bar \sigma_m\pi/m \}\cap\{  z=0\},
    \end{equation*}
    for some $\bar\sigma_m \in (0,1)$. 
    
    Furthermore, $m+1-\bar\mu_m\to 0^+$ and $\bar\sigma_m\to0^+$ as $m\to+\infty$ along odd numbers. 
\end{theorem}

\begin{figure}[h]
    \centering
    \subfigure[$m=3$]{
        \includegraphics[width=0.3\textwidth]{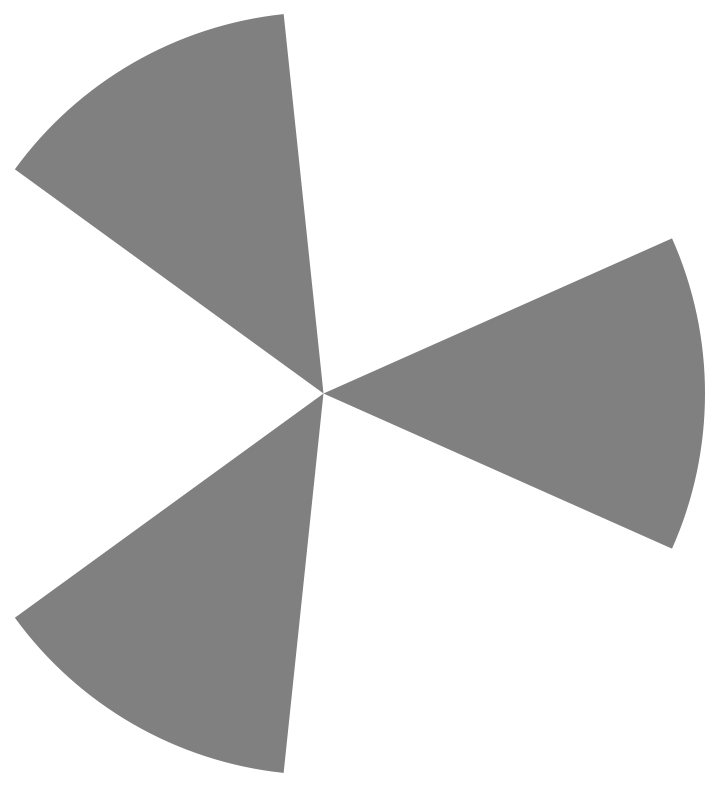}
    }
    \hfill
    \subfigure[$m=5$]{
        \includegraphics[width=0.3\textwidth]{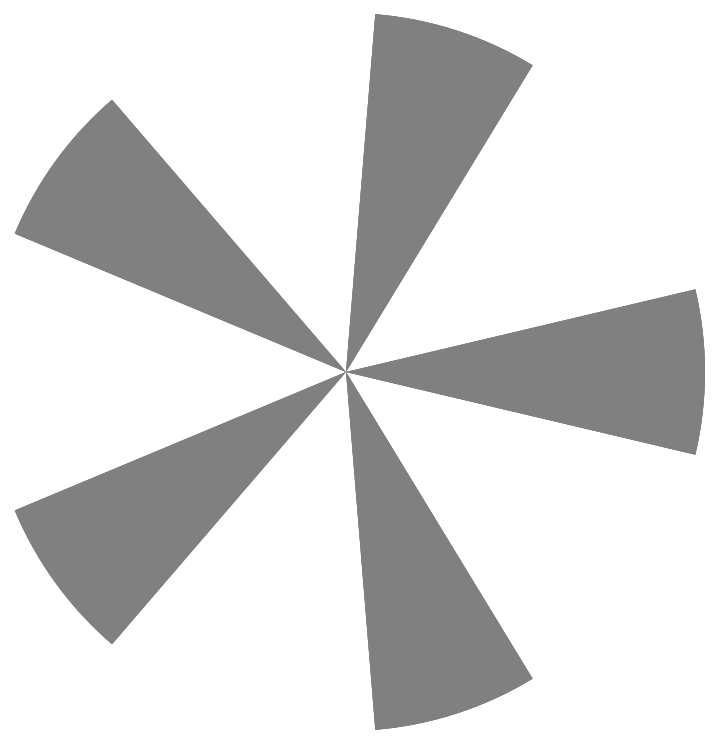}
    }
    \hfill
    \subfigure[$m=15$]{
        \includegraphics[width=0.3\textwidth]{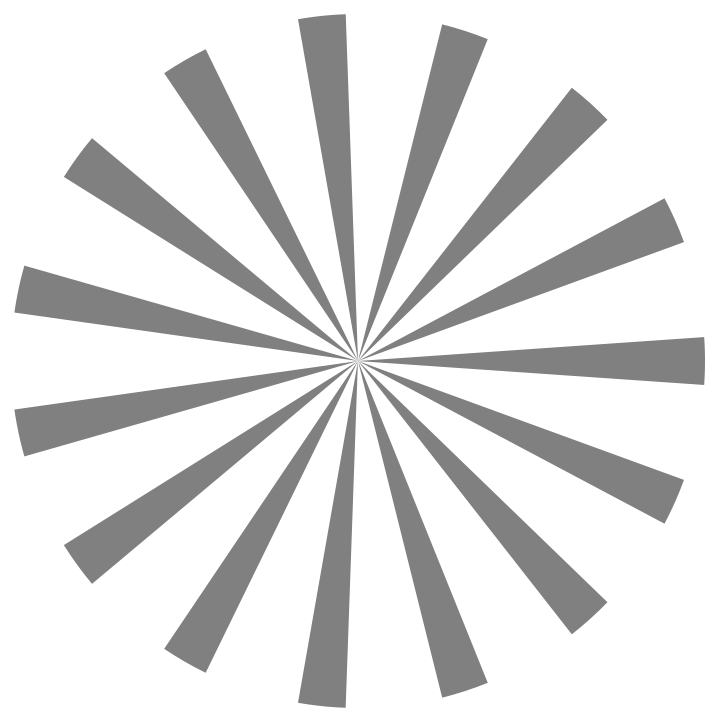}
    }
    \caption{The contact set $Z(u) = \{u =0 \}\cap\{ z=0\}$.}
    \label{fig:contactsets}
\end{figure}
\begin{remark}
Starting from a $\mu$-homogeneous solution in $n=3$, one can immediately construct a $\mu$-homogeneous solution for all $n\ge 3$.
\end{remark}

\begin{remark}
    Numerical simulations suggest
    \[
    \bar\sigma_3\in ({.39},{.42}{}),\qquad \bar\sigma_5\in ({.35},{.38}{}),\qquad \bar\sigma_{15}\in ({.31},{.34}{}),
    \]
    and
    \[
    \bar\mu_3 \in(3.52,3.59),\qquad \bar\mu_{5} \in (5.52,5.59),\qquad \bar \mu_{15} \in (15.60,15.67).
    \]
\end{remark}
We also prove a negative result in all dimensions, which shows that the restriction to odd values of $m$ is necessary.
\begin{theorem}\label{thm:gaps}
    Let $u\colon \R^n\to \R$ be a $\mu$-homogeneous solution of the thin obstacle problem, for some $n\ge1$. Then $\mu\notin \bigcup_{j\ge 0}(2j,2j+1).$
\end{theorem}
\subsection{Comments}
Our main theorems answer some open questions about the set of possible frequencies in the thin obstacle, that is the set
$$ \mathcal{F}_{nD} := \{ \mu > 0 : \text{there is a nontrivial solution to \eqref{eq:TOPsphere}} \}. $$
The only known examples of homogeneous solutions are based on extensions of the explicit solutions in dimension $n=2$. For all $n\ge2,$ it holds
\begin{equation}\label{1}
\mathcal{F}_{2D} = \mathbb{N} \cup \left \{ 2k+ \tfrac 32 : k \in \mathbb N \right \} \, \subseteq\, \mathcal{F}_{nD}.
\end{equation}
Athanasopoulous--Caffarelli--Salsa in \cite{ACS} classified the lowest frequencies and showed that for all $n\ge2,$
$$\mathcal{F}_{nD} \cap (0,2)= \{1,\tfrac 32 \},$$
which in turn gives the optimal regularity of solutions. 

By the work Garofalo--Petrosyan \cite{GP}, Colombo--Spolaor--Velichkov \cite{CSV}, Figalli--Ros-Oton--Serra \cite{FRS}, Savin--Yu \cite{SY, SY2}, when $\mu$ is an integer, then homogeneous solutions are classified and 
\begin{equation}\label{eq:intro1}
\mathcal{F}_{nD}\cap(m-c_{m,n},m+c_{m,n})=\{m\}\quad\text{for all}\quad m\in \N,
\end{equation} 
for suitable numbers $c_{m,n}>0$.


A dimension reduction argument based on Almgren's frequency, shows that the set $\mathcal{F}_{nD}\setminus\mathcal{F}_{2D},$ has dimension at most $n-3$ (see \cite{FoSpa}).
Beside this, very little was known for $\mathcal{F}_{nD}\setminus\mathcal{F}_{2D},$ not even if it was empty or not. 
This is the main novelty of \Cref{thm:main}: 
$$\mathcal{F}_{3D}\setminus\mathcal{F}_{2D}\neq\emptyset.$$

Additionally, \Cref{thm:main} shows that the numbers $c_{m,n}$ appearing in \eqref{eq:intro1} are necessarily infinitesimal as $m\to+\infty$ with $n$ fixed.

We also remark that, in all the known examples of homogeneous solutions, the boundary of the interior of the contact set was either empty or an $(n-2)$-dimensional linear space, in particular it was a flat cone.
\Cref{thm:main} shows that this set is not a flat cone in general. 

Beside its own interest in showing that the set $\mathcal{F}_{nD},$ has large gaps, \Cref{thm:gaps} has an important consequence for the free boundary regularity in the {\it thick} Obstacle Problem, by the work of Figalli--Serra \cite{FS} (see also Savin--Yu \cite{SY2,SY1}). 

Specifically, \Cref{thm:gaps} implies that the top-dimensional stratum of the singular set (the set $\Sigma_{n-1}$, in the notation of \cite{FS}) is locally contained in a $C^{1,1}$ hypersurface, rather than a $C^{1,\alpha_\circ}$ one. 
In other words, there are no ``anomalous'' points in $\Sigma_{n-1}$, see Remark 1.2 (2a) in \cite{FS} for more details.

\subsection{Connection with 2-valued harmonic functions and $\Z/2\Z$ eigenfunctions}

Our examples also work for unsigned versions of the thin obstacle problem.

Denote by $\mathcal{A}_2(\R)$ the space of unordered pairs of real numbers. If $u\colon\R^3\to\R$ is one of our examples constructed in \Cref{thm:main} for some $m$, then
\[
U\colon x\mapsto \{+u(x)\, ,\,-u(x)\}\ \in \ \mathcal{A}_2(\R),
\]
is a $\bar\mu_m$-homogeneous, 2-valued, $C^{1,1/2}$, stationary harmonic function (see Krummel--Wickramasekera \cite{KW}). 
These functions arise naturally as limits of stable minimal hypersurfaces, see Wickramasekera \cite{W}. 
We remark that $U$ is not a Dir-minimizer in the sense of Almgren, having a too big singular set, see De Lellis--Spadaro \cite{DLS}.

In a related framework, if we set
\[
\S^2\cap \Gamma(u) =: \{z_1,z_2,\ldots,z_{2m}\},
\]
then such a $U$ can be interpreted as an homogeneous example of an harmonic section of an appropriate real line bundle over $\S^2\setminus\{z_1,z_2,\ldots,z_{2m}\},$ which has nontrivial holonomy around each $z_i.$ See Taubes--Wu \cite{TW,TW2}, Donaldson \cite{D} and Haydys--Mazzeo--Takahashi \cite{HMT} for some recent developments.

\subsection{Structure of the paper}
In \Cref{sec:existence} we construct the solutions of \Cref{thm:main} and prove \Cref{thm:gaps}. The construction is based on the study of eigenfunctions in spherical domains inside the region $0\le \theta\le \pi/m,$ and it is short.        

The proof of the asymptotics of \Cref{thm:main} is carried out in \Cref{sec:asymptotics}. This requires a more refined analysis, but it is necessary in order to justify rigorously why the new frequencies are different from the 2D frequencies, at least for large values of $m$. We remark that this fact can be
inferred numerically already in the first few cases when $m= 3$, $m= 5$ or $m= 7.$

In \Cref{sec:variant} we present a natural variation of our construction, which allows us to construct even more solutions.

In \Cref{sec:legendre} we give a self-contained proof of some auxiliary results concerning Legendre functions.

\subsection{Acknowledgments} The first author gratefully acknowledges support from the Giorgio and Elena Petronio Fellowship while working on this project and thanks Camillo De Lellis for pointing out to him that the examples constructed here are not Dir-minimizers.

\section{Proof of \Cref{thm:main}}\label{sec:existence}
\subsection{Notation and Legendre functions}

The spherical coordinates in 3D are
\begin{equation*}
    x = r \cos \theta\cos\varphi,\quad y = r\sin\theta\cos\varphi,\quad z = r\sin\varphi,
\end{equation*}
with $\theta\in \R/2\pi\Z$ and $\varphi\in[-\pi/2,\pi/2].$ The thin space is $\{z=0\}=\{\varphi=0\}$.

We denote by $\S\subset\R^3$ the unit round sphere, by $N$ the north pole and $S$ the south pole and with $\S_+$ the open upper hemisphere $\S\cap\{z>0\}$. 

Let $\triangle_\S u$ be the spherical Laplacian, which in terms of $(\theta,\varphi)$ is expressed as
$$\triangle_\S u:= u_{\varphi \varphi} - \tan \varphi \, u_\varphi + (\cos \varphi)^{-2} \, u_{\theta \theta}.$$ 

We denote with $\rho$ the distance from the $z$ axis
\begin{equation*}
    \rho:=\sqrt{x^2+y^2}=\cos\varphi.
\end{equation*}

In the sphere $\S^{n-1}\subset\R^n$, the operator $L_\mu =\lap_{\S^{n-1}} +\mu(\mu+n-2)$ reads as 
\[
L_\mu w = \rho^{-2} \lap_{\S^{n-2}} w +  w_{\varphi \varphi} -(n-2)\tan \varphi\, w_\varphi +\mu(\mu+n-2) w.
\] 
We often write, for readability, 
\[
\lambda=\lambda_\mu=\mu(\mu+n-2). 
\]

For every $\mu \ge0 $ denote by $p_\mu\colon \S^{n-1}\setminus\{S\}\to \R$ the unique solution to
\begin{equation}
L_\mu p_\mu =0, \qquad p_\mu(N)=1
\end{equation}
which is also invariant by rotations around the $x_n$-axis. The function $p_\mu$ depends only on the angle $\varphi$, and it is uniquely determined by the final value problem
\begin{align}\label{2}
\begin{cases}
    p_\mu'' - (n-2) \tan \varphi \cdot p_\mu' + \mu(\mu+n-2)  p_\mu =0\quad \text {in}\quad[0,\tfrac\pi2],\\
    p_\mu(\tfrac\pi2)=1, \quad \quad p_\mu'(\tfrac\pi2)=0.
\end{cases}
\end{align}

If $\mu$ is an even (odd) integer, $p_\mu$ is a $\mu$-homogeneous harmonic polynomial even (odd) in $x_n$ and invariant by those rotations that fix $\{x_n=0\}$.

\begin{proposition}[Legendre functions]\label{prop:signs}
\,
\begin{itemize}
\item[(a)] The signs of $p_\mu(0)$ and $p_\mu'(0)$ are given by $\cos(\mu \frac \pi 2)$ and $\sin(\mu \frac \pi 2),$ respectively.
\item[(b)] If $m$ is odd and $\mu\in(m,m+1],$ then \begin{equation}\label{eq:6}
-p_\mu'(0)/p_\mu(0) \ge c m (m+1-\mu).
\end{equation}
\item[(c)] If $m$ is odd and $\mu\in[m+1-\delta,m+1],$ then
\begin{equation}\label{eq:7}
0\le-p_\mu'(0)/p_\mu(0)\le Cm (m+1-\mu).
\end{equation}
\end{itemize}
The positive constants $c$, $C$ and $\delta$ depend only on $n$.
\end{proposition}

Since it requires some computations, we postpone the proof of  \Cref{prop:signs}, which will be given in \Cref{sec:legendre} below.

We remark that in the case $n=3$, $p_\mu(\varphi) = P^0_\mu(\sin \varphi)$, where $P^\kappa_\nu(z)$ is the so-called associated Legendre function. Manipulating the explicit formulas available for such special functions, one could give an alternative proof of \Cref{prop:signs} (see equations 8.6.1 and 8.6.3 in \cite{stegun}). 

\begin{figure}[h]
    \centering
    \subfigure[$\mu=7$ in blue, $\mu=8$ in red]{
        \includegraphics[width=0.45\textwidth]{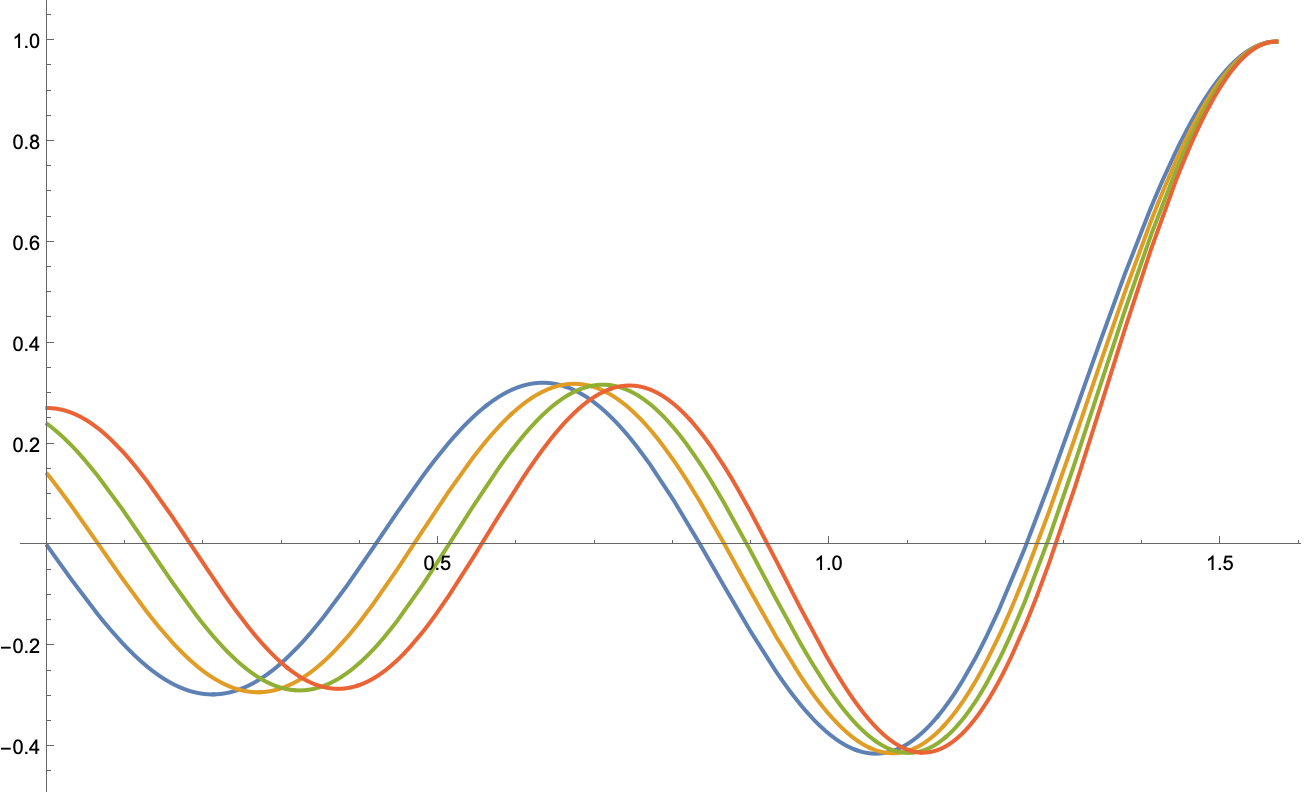}
    }
    \hfill
    \subfigure[$\mu=12$ in blue, $\mu=13$ in red]{
        \includegraphics[width=0.45\textwidth]{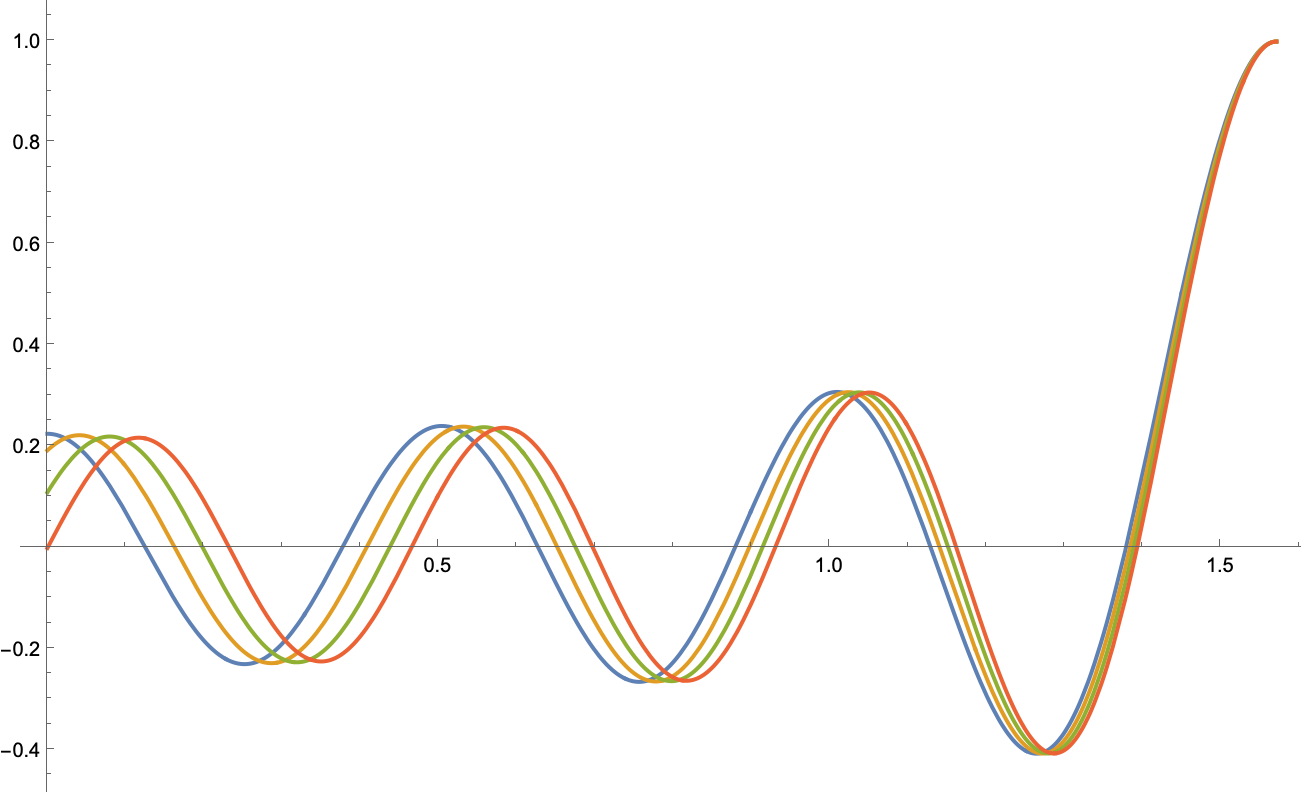}
    }
    \caption{Plots of $p_\mu(\varphi)$ as $\mu$ ranges between two consecutive integers and $\varphi\in[0,\tfrac{\pi}{2}]$.}
    \label{fig:pmu}
\end{figure}

\subsection{Construction of $\bm {v,}$ $\bm{h,}$ and $\bm u$}\label{subsec:vhu}
For each integer $m \ge 1$  and $\sigma\in[0,1]$ consider the spherical closed cone in $\R^3$
\begin{equation}\label{eq:defSigma}
\Sigma=\Sigma_{m,\sigma}:=\big\{\dist(\theta,\tfrac{2\pi}{m}\Z )\le \sigma\pi/m \ ,\  z=0\big\}
\end{equation}
and the spherical domain
\begin{equation*}
    \Omega=\Omega_{m,\sigma}:=\{-\tfrac\pi2<\varphi<\tfrac\pi2\ ,\ 0<\theta<\pi/m\} \setminus \Sigma_{m,\sigma} \quad \subset\quad\S.
\end{equation*}
which corresponds to the region $\theta \in (0, \pi/m)$ from which we remove a left segment of length $\sigma \pi/m$ from the equator.

We denote by $v = v_{m,\sigma} \ge 0$ the first Dirichlet eigenfunction of $-\triangle_\S$ in the spherical domain $\Omega$, which is also even with respect to $\varphi$.

Denote by $\mu=\mu_{m,\sigma}>0$ the respective value which makes the $\mu$-homogeneous extension of $v$ harmonic in $\R^3$.

Notice that $\mu$ is uniquely determined by the pair $(m,\sigma)$ and it is a continuous and strictly increasing function of $\sigma\in[0,1]$. 

We construct $u$ from $v$ in requiring that
\[
\p_\theta u = v.
\]
That is to say
$$u(\theta,\varphi):= \int_0^\theta v(s,\varphi) ds + h(\varphi)\quad \text{for}\quad (\theta,\varphi) \in [0, \pi/m]\times [0,\pi/2),$$ 
and then extend $u$ by even reflections (both in $\theta$ and $\varphi$) to the whole unit sphere. 
Using that $L_\mu v =0$, we find 
$$L_\mu u= L_\mu h + \rho^{-2} \,v_\theta(0,\varphi),$$
so we choose $h\colon [0, \frac \pi 2) \to \R$ as the solution of the initial value problem:
\begin{equation}\label{eq:ODE}
\begin{cases}
    L_\mu h + (\cos \varphi)^{-2} \, v_\theta(\varphi,0)=0, \\
    h(0) =0,\\
    h'(0) = - \int_0^{\pi/m} v_\varphi(s,0^+) ds. 
\end{cases}
\end{equation} 
We used the notation $h'=h_\varphi$. Since $u$ solves $L_\mu u=0$ in $\S_+\setminus\{N\}$ and $h$ may blow-up at most logarithmically around $N$ (see \Cref{lem:logbound}), we must have
\begin{equation*}
    L_\mu u = c(\sigma) \delta_N \quad \quad \text{in} \quad \S_+ = \S \cap \{z>0\},
\end{equation*}
where $\delta_N$ denotes the Dirac-delta measure at the north pole $N$, 
and $c(\sigma)$ is a constant that depends on the eigenfunction $v$. 

The boundary conditions for $h$ give that, for all  $\theta\in(0,\pi/m)$,
\begin{equation}\label{eq:uequator}
           u(\theta,0)=  \int_{\sigma\pi/m}^{\max\{\theta,\sigma\pi/m\}} v(s,0)\, ds,\quad u_\varphi (\theta,0+) = - \int_{\min\{\theta,\sigma\pi/m\}}^{\sigma\pi/m}  v_\varphi(s,0+)\,ds.
\end{equation}
We used the notation $v_\varphi(s,0+) = \lim_{t\to 0^+}v_\varphi(s,t)$.
Since $v\ge 0$ we immediately find
\begin{itemize}
    \item $u = 0$ and $u_\varphi\le0$ on the segment $[0,\sigma\pi/m]\times \{0\}$,
    \item $u \ge 0$ and $u_\varphi =0$ on the remaining part of the equator $[\sigma\pi/m,\pi/m]\times\{0\},$
    \item $u_\theta =0$ on the boundaries $\theta=0$ and $\theta=\pi/m$.
\end{itemize}
This proves that $u$ solves the Thin Obstacle Problem \eqref{eq:TOPsphere} if and only if $c(\sigma)=0$. 

We express the number $c(\sigma)$ in terms of $v$ and the Legendre functions $p_\mu$. Integrating by parts $p_\mu \cdot L_\mu u$ the upper sphere we find
\begin{align}\label{eq:intbyparts}
    \frac{c(\sigma)}{2 m}&=\frac{1}{2m}\int_{\S_+} p_\mu L_\mu u = - p_\mu(0)\int_0^{\pi/m}u_\varphi (t,0+)\,dt + p'_\mu(0)\int_0^{\pi/m}u(t,0)\, dt.
\end{align}
So, using \eqref{eq:uequator} to express $u$ and $u_\varphi$ in terms of $v$ and $v_\varphi$ on the equator,
\begin{equation}\label{eq:quantity}
\frac{c(\sigma)}{2 m}=p_\mu(0) \int_0^{\pi/m} v_\varphi(s,0+) s\,ds + p_\mu'(0) \int_{0}^{\pi/m} v(s,0) (\tfrac{\pi}{m}-s )\, ds.
\end{equation}

\begin{figure}[h]
    \centering
    \subfigure[$\sigma<\bar\sigma_3$, positive pole at $\varphi=\tfrac\pi2.$]{
        \includegraphics[width=0.45\textwidth]{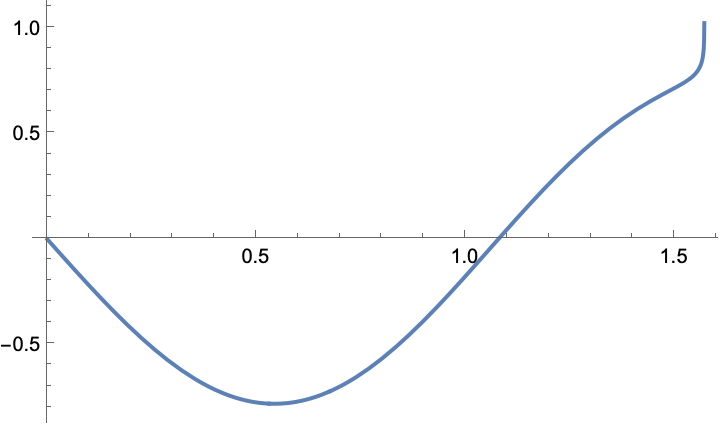}
    }
    \hfill
    \subfigure[$\sigma > \bar\sigma_3$, negative pole at $\varphi=\tfrac\pi2.$]{
        \includegraphics[width=0.45\textwidth]{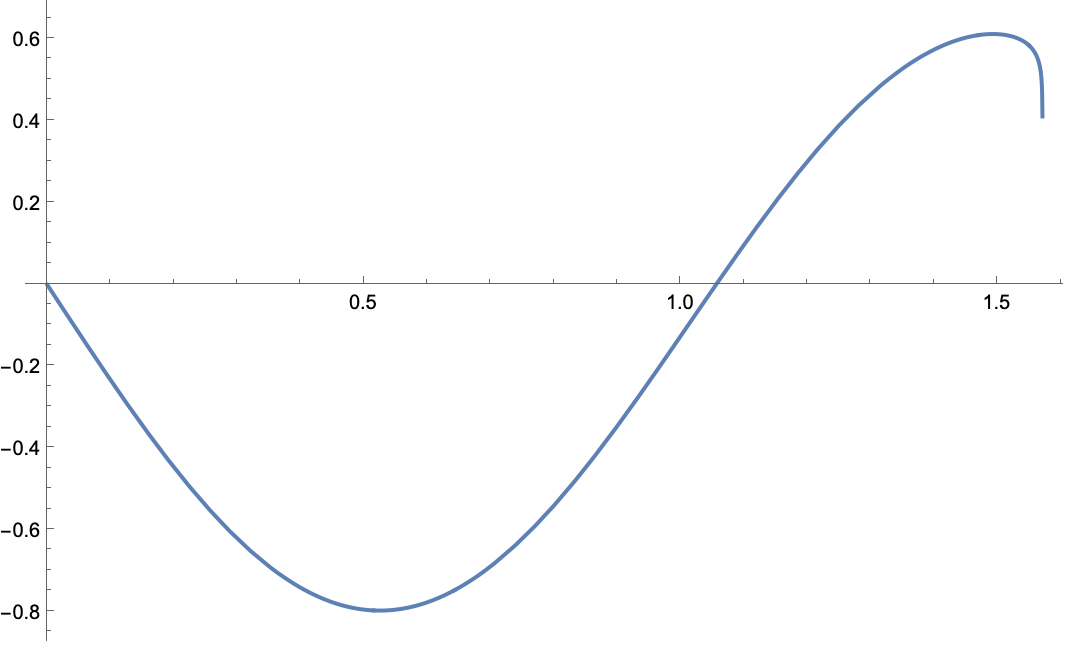}
    }
    \caption{Numerical plots of $h$ for $m=3$ and different values of $\sigma$.}
    \label{fig:h}
\end{figure}

\subsection{ Existence of solutions }\label{subsec:existence}
Fix any odd $m>0$. 
If we pick the $v_{m,\sigma}$ to be positive and with maximum value $1$, the map $\sigma\mapsto v_{m,\sigma}$ becomes continuous and so does the map $\sigma\mapsto c(\sigma)$. 
Thus, in order to prove the existence of $\bar\sigma_m\in(0,1)$ and $\bar\mu_m,$ it is enough to show that $c_m(0)\cdot c_m(1)<0$. 

Now the functions $v_{m,0},v_{m,1}$ and their eigenvalues are completely explicit. For $\sigma=0$ we have
\begin{equation*}
    v = \rho^m \sin(m\theta)=\im(x+iy)^m>0,\qquad \mu=m,\qquad  v_\varphi(\theta,0)=0,
\end{equation*}
and so
\begin{equation*}
     c_m(0)=p_m'(0)\int_0^{\pi/m}\sin(ms)(\tfrac\pi m-s)\, ds.
\end{equation*}
For $\sigma=1$ we have
\begin{equation*}
    v = z\rho^m \sin(m\theta)=z\im(x+iy)^m>0,\qquad \mu=m+1,\qquad  v(\theta,0)=0,
\end{equation*}
and so
\begin{equation*}
     c_m(1)=p_{m+1}(0)\int_0^{\pi/m}\sin(ms)\,s\, ds.
\end{equation*}
Then it is clear that $c_m(0)\cdot c_m(1)$ has the same sign as $p_m'(0)\cdot p_{m+1}(0).$ By \Cref{prop:signs} when $m$ is odd we have that (up to multiplications by positive numbers)
\[
p_m(0)=0,\qquad p_m'(0)=\pm1,
\]
and 
\[
p_{m+1}(0)=\mp 1,\qquad p_{m+1}'(0)=0,
\]
so $c_m(0)\cdot c_m(1)<0$.
Finally, $\bar\mu_m\in (m,m+1)$ by monotonicity of $\sigma\mapsto \mu_{m,\sigma}$.

\begin{figure}[h]
    \centering
    \subfigure{
        \includegraphics[width=0.45\textwidth]{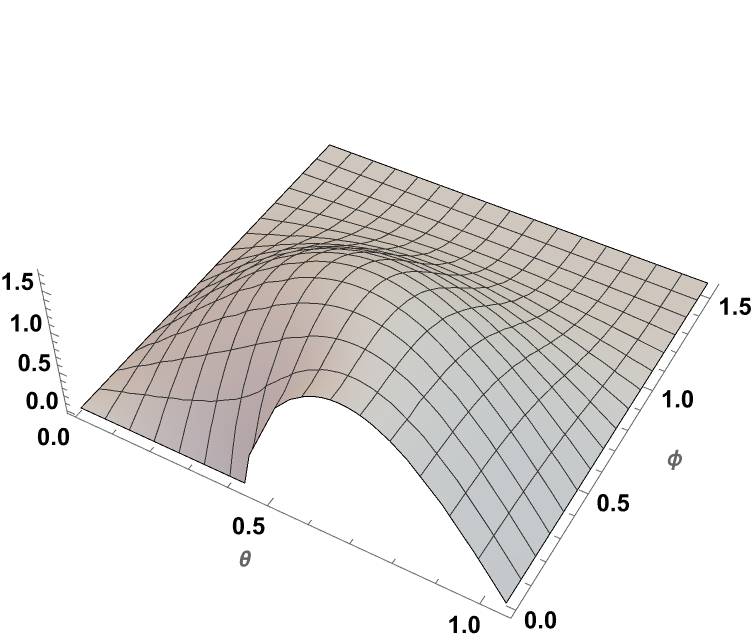}
    }
    \hfill
    \subfigure{
        \includegraphics[width=0.45\textwidth]{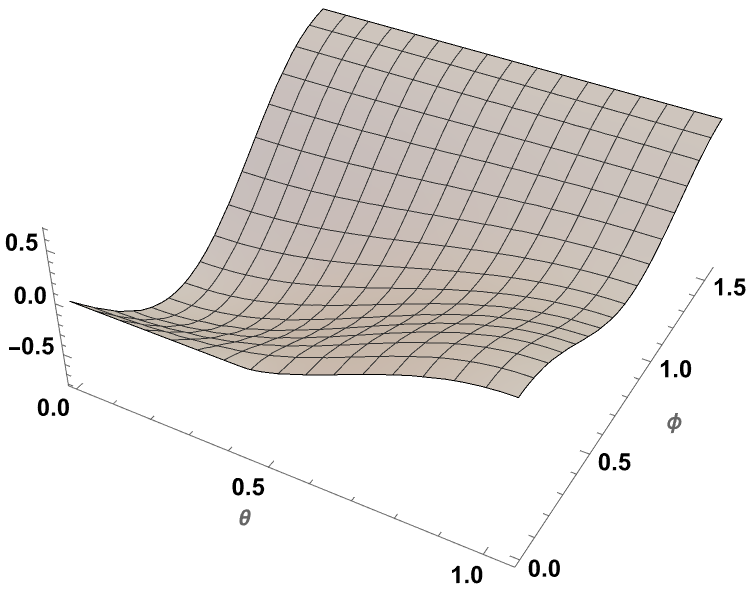}
    }
    \caption{Plots of $v$ (left) and $u$ (right) for $m=3$ and $\sigma=.42$}
    \label{fig:uv}
\end{figure}

It is natural to inspect this computation in the case of even $m$.
\subsection{Non-existence of solutions --- Proof of \Cref{thm:gaps}}

By \Cref{prop:signs}, when $m$ is even we have, up to positive multiplicative constants,
\[
p_m(0)=\pm 1,\qquad p_m'(0)=0,
\]
and 
\[
p_{m+1}(0)=0,\qquad p_{m+1}'(0)=\pm 1,
\] 
and 
\begin{equation}\label{eq:i}
p_\mu(0) \cdot p_\mu'(0) >0\quad\text{for all}\quad\mu\in(m,m+1).
\end{equation}
Comparing this with \eqref{eq:intbyparts} one finds that our examples cannot be constructed for even $m$'s. Actually, much more is true, since computation \eqref{eq:intbyparts} only used that $u$ was an homogeneous solution. This leads to \Cref{thm:gaps}.

\begin{proof}[Proof of \Cref{thm:gaps}]
Assume that $u$ is a $\mu$-homogeneous solution of \eqref{eq:TOPsphere} in $\S^{n-1}$ and that $\mu\in(2k,2k+1)$ for some integer $k$. We perform the same computation that gave \eqref{eq:intbyparts}:
that is we integrate by parts $p_\mu \cdot L_\mu u =0$ on $\S_+^{n-1}$. We find
\begin{equation}\label{3}
 p'_{\mu}(0) \int_{\S^n \cap \{x_n=0\} } u -p_\mu(0) \int_{\S^n \cap \{x_n=0\} }u_{n} =0.    
\end{equation} 
By \eqref{eq:i} and the sign conditions on $u$ and $u_n$, we must have that both $u$ and $u_n$ vanish identically on the equator, which implies that $u$ is zero.
\end{proof}
%
%
%
%

\section{Asymptotics for ${m\to+\infty}$}\label{sec:asymptotics}
We turn to the proof that $\bar\mu_m-m\to1$ and $\bar\sigma_m\to 0$ as $m\to+\infty$ along odd numbers. 

In order not to deal with $\pm$ symbols, we assume in the rest of the section that $m$ is congruent to $3$ modulo $4,$ so that
\[
p_\mu(0)>0,\quad p'_\mu(0)<0\quad\text{for all}\quad\mu\in(m,m+1).
\]
The other case is completely analogous.

For each $m$ and $\sigma\in[0,1]$ we take $v=v_{m,\sigma}$ as in the previous section and rescale
\[
\tilde v (x,y) = v(x/m,y/m),\qquad \tilde p_\mu(y):=p_\mu(y/m),
\]
with $x\in[0,\pi],\ y\in[0,m\pi/2]$. 
Multiplying $v$ by a positive constant, we may also assume that
\[
\int_0^\pi \tilde v(x,1)\sin x\, {dx} =\frac\pi2.
\]

Now equation \eqref{eq:quantity} gives that $c(\sigma)=0$ if and only if \begin{equation}\label{eq:c=0}
\tilde p_\mu(0)\int_0^{\pi} \tilde v_y(x,0+) x\,dx +{\tilde p'_{ \mu}(0)} \int_{0}^{\pi} \tilde v(x,0) ({\pi}-x )\, dx=0.
\end{equation}

We show that, if $m<\mu\le m+ 1-\eta$ while $m\to+\infty,$ then
\begin{equation}\label{eq:bw1}
\tilde v\to  \sin{x} \qquad  \text{and}\qquad\liminf_{m}\frac{-\tilde p_\mu'(0)}{\tilde p_\mu(0)}\ge c_\eta,
\end{equation}
for some small $c_\eta>0$.
This is incompatible with \eqref{eq:c=0}, proving $\bar \mu_m-m\to 1^-$.

Similarly, if $\sigma\ge\eta>0$ and $\mu-m\to1,$ while $m\to+\infty$, then
\begin{equation}\label{eq:bw2}
\int_0^\pi \tilde v_y(x,0^+)(\pi -x)\ge c\eta^2 \qquad  \text{and}\qquad \limsup_{m}\frac{-\tilde p_\mu'(0)}{\tilde p_\mu(0)}=0, 
\end{equation}
which is again incompatible with \eqref{eq:c=0}, proving $\bar \sigma_m\to 0$.

In the rest of this section, we prove \eqref{eq:bw1} and \eqref{eq:bw2}.  
\begin{itemize}
\item The claimed asyptotics for $-{\tilde p_\mu'(0)}/{\tilde p_\mu(0)}$ are direct consequences of items (b) and (c) of \Cref{prop:signs}.
\item Using a small multiple of $y \sin{x} $ as lower barrier for $\tilde v$, one finds immediately that for $m$ large and $c$ universal:
\[
\int_0^\pi \tilde v_y(x,0^+) (\pi -x)\ge c\int_0^\eta  (\pi-x)\, \sin x \,dx\ge c \eta^2.
\]
Indeed, if we express $L_\mu$ in the $(x,y)$ coordinates we find 
\begin{align}\label{eq:Ltilde}
\tilde L_\mu  &= \frac{1}{\cos^2(y/m)}\p_{xx} +\p_{yy} -\frac{\tan(y/m)}{m} \p_y +\frac{\mu(\mu+1)}{m^2},
\end{align}
thus for $(x,y)\in [0,\pi]\times[0,1]$ we have
\begin{align*}
	\tilde L_\mu (y\sin x) & = \Big(\frac{\mu(\mu+1)}{m^2}  -\cos(y/m)^{-2} -\frac{1}{m^2}\frac{\tan(y/m)}{y/m}\Big) y \sin{x} \\
	&\ge \big(1+\tfrac1m -1 - O(1/m^2)\big)y \sin{x} \ge0,
\end{align*}
at least for for $m$ large.
\item If we known that $\tilde v\to  \sin{x} $ uniformly in $Q:=[0,\pi]\times[0,1]$, then 
\[
\int_0^\pi \tilde v_y(x,0^+)x\, dx \le \pi\int_0^\pi \tilde v_y(x,0^+)\, dx \to 0.
\]
Notice that by elliptic estimates $\tilde v\to  \sin{x} $ in $C^1_\loc([0,\pi]\times (0,1])$ (the coefficients of $\tilde L_\mu$ have uniform bounds in $m$).
Using the divergence theorem, we can express $\int_0^\pi \tilde v_y(x,0^+)\, dx $ as sum of
\[
\int_0^1 \tilde v_x(0,y)\cos(y/m)\, dy, \quad \int_0^1 \tilde v_x(\pi,y)\cos(y/m)\, dy,\quad  \frac{\lambda_\mu}{m^2}\int_Q \tilde v \cos(y/m).\]
The conclusion follows by the dominated convergence theorem. Indeed a large multiple of $\cos(y) \sin{x} $ is an upper barrier for $\tilde v$ in $Q$, so for some universal  constant $C,$ uniform in $m,$ on the lateral sides we have 
\[
|\tilde v_x(0,y)| + |\tilde v_x(\pi,y)|\le C\quad \text{for all}\quad y\in[0,1].
\]
\end{itemize} 

So the proof of \eqref{eq:bw1} and \eqref{eq:bw2} is reduced to the proof of the following Proposition.

\begin{proposition}[Blow-up]\label{prop:blowup} Take any sequence $\sigma_m\in(0,1)$ and construct the relative $\mu_m,$ $v_m$ and $\tilde v_m$. Assume that
\[
\limsup_m \, (\mu_m-m)<1.
\]
Then $\tilde v_m\to  \sin{x} $ uniformly on compact sets of $[0,\pi]\times[0,+\infty)$.
\end{proposition}

%
%
%


We need some preliminaries.
In order to understand $v$ for large $m$, we compare it with $V>0$ which solves the same equation $L_\mu V=0$ with Dirichlet data on the larger domain
\[
\Omega':=\{0<\theta<\pi/m, -a_\mu<\varphi <\pi/2\}\subset \S.
\]
We also normalize $V$ in such a way that
\begin{equation}\label{eq:normalization}
\int_0^\pi \tilde v(x,1)\sin x\, {dx} =\int_0^\pi \tilde V(x,1)\sin x\, {dx} =\frac\pi2.
\end{equation}

Notice that the length $a_\mu\in(0,\tfrac\pi2)$ is uniquely determined by $\mu$ (or $m,\sigma$) and $a_m=0$ and $a_{m+1}=\pi/2$. Thus $a_\mu$ captures the eigenvalue defect $\mu-m$. On the other hand we show that, in the case $\mu\le m+1-\eta$, we have
\[
a_\mu\ge c_\eta\sqrt{m},
\]
which implies that, in the $1/m$-rescaling around the equator, $a_\mu$ escapes to $+\infty$ and this will force $\tilde V\to \sin{x} $.

This is the contenent of \Cref{lem:blowupV}. It may be helpful to keep in mind that, by separation of variables, $V=w(\varphi)\sin(m\theta)$ where $w$ solves the equation\footnote{\ In fact, up to a multiplicative constant, $w(\varphi)=P^m_\mu(\sin\varphi),$ where $P^\kappa_\nu$ is an associated Legendre function, see \cite{stegun}.}
\begin{equation}\label{eq:associatedlegendre}
w'' -\tan\varphi \, w' +\mu(\mu+1)w=\rho^{-2}{m^2} w\quad\text{in}\quad(-a_\mu,\tfrac\pi2).
\end{equation}

\begin{lemma}[Properties of $V$]\label{lem:blowupV} Given a sequence $\sigma_m\in (0,1)$, construct the relative $\mu_m$ and $\tilde V_m$ as above. Assume that there is $\eta>0$ such that
\[
\mu_m\le m+1-\eta\quad \text{for all}\quad m.
\]
\begin{itemize}
\item[(a)] There is $c_\eta>0$ such that
\begin{equation*}
a_\mu\ge c_\eta/\sqrt m.
\end{equation*}
\item[(b)] $\tilde V_m\to  \sin{x} $ uniformly on compact sets of $[0,\pi]\times[0,\infty)$.
\end{itemize}
\end{lemma}
\begin{proof}
We call universal constants that are uniform in $m$ and $\sigma$.
We also drop the subscript $m$ from $V_m$.

We argue by blow-up. In both (a) and (b) we will need the following barrier from above:
\begin{equation}\label{eq:V1def}
 V_1:= z\rho^m \sin(m\theta)  \ge 0\quad \text{in}\quad\Omega,
\end{equation}
which satisfies $L_{m+1}V_1=0,$ so $L_\mu V_1\le0$ in $\Omega$.

We first prove (a).

We consider $\hat V$ to be the multiple of $V$ such that 
$$\fint_0^{\pi/m}\hat V(\theta,1/\sqrt{m})\, d\theta= 1.$$
Since in the box of sidelegth $1/m$ and latitude $\varphi=1/\sqrt m$, the rescaled operator $\tilde L_\mu$ still looks like a $O(1/m)$-perturbation of $\lap+1$, there is a universal constant $C$ such that
\[
\hat V \le C\sqrt m V_1\quad\text{on}\quad[0,\pi/m]\times \{\tfrac{1}{\sqrt m}\}.
\]
Hence the maximum principle gives
\begin{equation}\label{eq:barriersV2}
\hat V(\theta,\varphi)\le C\sqrt m  V_1(\theta,\varphi)\quad \text{in}\quad\{\varphi\ge1/\sqrt m\}\cap\Omega.
\end{equation}
Now we consider the rescaled Fourier coefficient
\[
\tilde w(t/\sqrt m):= \fint_0^{\pi/m} \hat V(\theta,t/\sqrt m)\sin(m\theta)\, d\theta,\qquad t\in(-\sqrt m \cdot a_\mu,\sqrt m\cdot  \tfrac\pi2),
\]
which, setting $\mu:=m+\delta,$ solves in that interval the equation (see \eqref{eq:associatedlegendre})
\begin{equation}\label{eq:a}
\tilde w '' -\frac{\tan(t/\sqrt m)}{\sqrt m}\tilde w' + \big\{\frac{(m+\delta)(m+1+\delta)}{m}-\frac{m}{\cos^2(t/\sqrt m)}\big\}\tilde w =0,
\end{equation}
with the conditions 
\begin{equation}\label{eq:b}
\tilde w (1) =1,\quad \tilde w(-\sqrt m a_\mu)=0,\quad \tilde w>0.
\end{equation}
Additionally, rescaling and integrating \eqref{eq:barriersV2} we find for $m$ large
\begin{equation}\label{eq:g}
 \tilde w\le C t e^{-t^2/2}\quad \text{for all}\quad 1\le t\le \tau_m,
\end{equation}
for a universal sequence $\tau_m\to+\infty$. Indeed, using the explicit formula \eqref{eq:V1def}, it is readily checked that
\[
\sqrt m \fint_0^{\pi/m} V_1(\theta,t/\sqrt m)\sin(m\theta)\, d\theta \to te^{-t^2/2}\quad\text{as}\quad m\to+\infty,
\]
locally uniformly in $t$. By elliptic estimates on \eqref{eq:a}, \eqref{eq:g} also gives 
\begin{equation}\label{eq:derivativebound}
|\tilde w'(1)|\le C.
\end{equation}

Let's assume by contradiction that (a) fails for some $\eta>0$, i.e.,
$$\mu_m-m\to \bar\delta\le1-\eta,\quad \text{but}\quad a_\mu \sqrt m \to0.$$ 

Putting together \eqref{eq:a}, \eqref{eq:b}, \eqref{eq:derivativebound} and \eqref{eq:g} and letting $m\to+\infty$ we find by compactness a function $\bar w\colon[0,+\infty)\to[0,+\infty)$ which solves 
\[
\bar w''(t) + (1+2\bar \delta-t^2)\, \bar w(t)=0,\quad \bar w(0)=0, \quad \bar w\ge 0,\quad \bar w(1)=1,
\]
and
\[
\bar w(t) \to 0\quad\text{as}\quad t\to+\infty.
\]
By the ODE, $\bar w$ is eventually convex, positive and infinitesimal, so $\bar w'(t)$ is infinitesimal as well.
The existence of such a $\bar w$ is incompatible with $\bar\delta <1$. Indeed, consider the first eigenfunction $\psi_1(t):=te^{-t^2/2}$ which solves 
\[
\psi_1 '' + (3-t^2) \psi_1 =0,\quad \psi_1(0)=0, \quad \psi_1'(+\infty)=\psi_1(+\infty)=0.
\]
Multiply this equation by $\bar w$ and integrate it in the interval $[0,L]$, finding
\begin{align*}
0&= \int_0^L \bar w \psi_1'' + \int_0^L \bar w (3-t^2) \psi_1 \\
&= \bar w (L) \psi_1'(L)-\bar w'(L) \psi_1(L) +2(1-\bar\delta)\int_0^L\bar w \psi_1.
\end{align*}
Since the boundary terms are infinitesimal as $L\to+\infty,$ we find a contradiction. This concludes the proof of part (a).

We turn to the proof of part (b).

Recalling the normalization \eqref{eq:normalization} and arguing as in the proof of (a), but rescaling this time by $1/m$ rather than $1/\sqrt m$, we find that
\begin{equation}\label{eq:c}
 V(\theta,\varphi) \le C\,  m\, V_1(\theta,\varphi)\quad\text{in}\quad\{\varphi\ge1/ m\}\cap\Omega,
\end{equation}
for some universal $C$.
This, together with the Harnack inequality for $\tilde L_\mu$ and part (a), gives the rough bound:
\begin{equation}\label{eq:sublinear}
\tilde V (x,y) \le C(y+e^{-Cy})\quad \text{for all}\quad (x,y)\in[0,\pi]\times[-c_\eta \sqrt m, m\tfrac\pi4].
\end{equation}
Which, by elliptic estimates on $\tilde L_\mu$ and part (a), implies that  
$$\tilde V\to \bar V$$ 
locally uniformly in the strip $[0,\pi]\times\R$, with 
\begin{equation*}\label{eq:pdeV}
\lap \bar V+\bar V = 0,\quad  \bar V\ge0,\quad   \bar V=0\quad\text{on}\quad\{0,\pi\}\times (0,\infty),\quad \int_0^\pi \bar V(x,1) \sin{x} \,dx=\frac\pi2.
\end{equation*}
These conditions imply that $\bar V= \sin{x} .$ Indeed, the Fourier coefficient 
$$f(y):=\int_0^\pi \bar V(x,y) \sin{x} \, dx,$$
must be constantly equal to $f(1)=\pi/2$, because $f\ge 0$ and $f''=0$ on the whole real line.
This implies that $\bar V$ is uniformly bounded above in the full strip, and by \Cref{lem:uniqueness} we conclude that $\bar V= \sin{x} .$
\end{proof}
The Boundary Harnack principle implies that $\tilde V$ and $\tilde v$ have the same behaviour as $y\to+\infty$.
\begin{lemma}[$v$ and $V$]\label{lem:bh}
    For all $m\ge1 $ and $\sigma\in[0,1]$ we have $c\le l\le C$ such that
    \begin{align}
        \big|\frac vV-l|&\le Ce^{-cm\varphi}\quad \text{in}\quad\Omega\cap\{\varphi\ge1/m\}, \label{eq:consequencebh}
    \end{align}
where $v,V$ are normalized as in \eqref{eq:normalization}.

The constants $C$ and $c$ are independent from $m,\sigma$.\end{lemma}
\begin{proof}
    We call universal constants that are uniform in $m$ and $\sigma$.
    
    \textbf{Claim.} Assume that for some positive numbers $a\le b$ and some $\beta \in [\tfrac4m ,\tfrac{\pi}{3}]$ we have
    \begin{equation*}
        aV \le v\le bV \quad \text{ in }\quad \{\varphi\ge \beta-\tfrac1m\}\cap\Omega.
    \end{equation*}
    Then there are $a\le a'\le b'\le b$ such that
    \begin{equation*}
        a'V \le v\le b'V \quad \text{ in }\quad \{\varphi\ge \beta\}\cap\Omega,
    \end{equation*}
    with
    \begin{equation*}
        b'-a'\le (1-c_0)(b-a),
    \end{equation*}
    for some universal $c_0>0$ small.

    We show how to conclude from the claim. Assume
    \[
    a_0\le \frac{v}{V} \le b_0\quad \text{ in }\quad{[0,\tfrac{\pi}{m}]\times \{\tfrac{1}{m}\}},
    \]
    then an iteration of the claim gives that for some $l\in [a_0,b_0]$ and all $k\le {m\pi}/{3}$, it holds 
    \[
 \Big|\frac vV-l| \le (1-c_0)^k(b_0-a_0)\quad \text{in}\quad{[0,\tfrac{\pi}{m}]\times [\tfrac 1m +\tfrac k m ,\tfrac\pi2)}.
    \]
    For suitable universal constants $C$ and $c$ this becomes
    \[
    \Big|\frac vV -l\Big| \le Ce^{-c m \varphi}(b_0-a_0)\quad\text{in}\quad\{\varphi \ge\tfrac 1 m \},
    \]
    which is the sought inequality.
    
    {\it Proof of the Claim.}
        Set $\delta:=b-a.$ 
        Denote by $q$ the point on the sphere with coordinates $(\theta,\varphi)=(\tfrac{\pi}{2m},\beta)$.
        Look in the cube $Q:=\{|m\theta-\tfrac\pi2|\le \tfrac\pi4, |\varphi-\beta|\le \tfrac{7}{8m} \}$, either
        \begin{equation}\label{eq:either}
        v-b V\le -c_1\delta V \quad \text{in}\quad Q,
        \end{equation}
        or 
        \[
        (v-a V)(q')\ge (1-c_1)\delta V(q') \quad \text{ for some }q'\in Q ,
        \]
        which, by the Harnack inequality in $Q$, becomes
        \begin{equation}\label{eq:or}
      v-aV \ge c_2 \delta V\quad \text{in}\quad Q,
        \end{equation}
        As before, the constants $c_1,c_2$ are uniform in $m$ because the rescaled operator $L_\mu$ is a $O(1/m)$-perturbation of $\lap+1$.
        In either cases, one can use the regularity of $V$ and the Hopf Lemma on $v-aV$ (or $bV-v$) to move \eqref{eq:either} (or \eqref{eq:or}) from $Q$ to the whole segment $[0,\pi]\times\{\beta\}$. And then, by the maximum principle for $L_\mu$, one can extend \eqref{eq:either} (or \eqref{eq:or}) to the whole $\Omega\cap \{\varphi\ge \beta\}.$ This proves the claim. 
\end{proof}
We combine these two Lemmas to show that any blowup of $\tilde v$ --- which solves $\lap+1=0$ on the strip ---  has $ \sin{x} $ as boundary condition at $y=+\infty$. With positivity this is enough to conclude that it must be equal to $ \sin{x} $, proving \Cref{prop:blowup}. 

\begin{proof}[Proof of \Cref{prop:blowup}]
In this proof $c$ and $C$ denote constants which are uniform in $m$ and $\mu$, we also drop the subscript $m$ from $\tilde v_m$ and assume $\tilde v$ is reflected evenly in the $y$ variable.

As in \Cref{lem:blowupV} for $V$, the normalization gives \begin{equation}\label{eq:pippo}
    c \sin{x} \le \tilde v(x,1)\le C \sin{x} \quad \text{for all}\quad x\in[0,\pi],
\end{equation}
and with the barrier $\sin{x}\cdot\cos y $ one finds
\[
 \tilde v \le C\quad\text{in}\quad[0,\pi]\times[-1,1].
\]
We also have that 
\begin{equation}\label{eq:holderbound}
\|\tilde v\|_{C^\alpha([0,\pi]\times[-1,1])}\le C,
\end{equation}
see the end of the proof for a quick proof of this bound.

By \Cref{lem:bh} there is some constant $l\in(c,C),$ such that
\begin{equation}\label{eq:kl}
    \Big|\frac{v(\theta,\varphi)}{ V(\theta,\varphi)} -l\big|\le Ce^{-c m\varphi}\quad \text{for all}\quad1/m\le\varphi\le\pi/2,
\end{equation}
so, using elliptic estimates for $\tilde L_\mu$ and \eqref{eq:sublinear}, this implies $C^1_\loc([0,\pi]\times(0,+\infty))$ bounds for $\tilde v$. Putting things together:
\[
\tilde v \to \bar v\qquad\text{locally uniformly in the full strip.}
\]

Rescaling by $1/m$, passing to the limit \eqref{eq:kl} and using \Cref{lem:blowupV} (since $\lim_m\mu-m<1$) we find
\[
|\bar v(x,y)-l \sin{x} |\le Ce^{-c y}  \sin{x} \quad \text{for all}\quad y\ge1,
\]
so $\bar v$ is bounded. In conclusion, $\bar v$ solves
\begin{equation}\label{eq:9}
\begin{cases}
    \lap \bar v+\bar v=0& \text{ in }(0,\pi)\times (0,\infty),\\
     \bar v>0& \text{ in }(0,\pi)\times (0,\infty),\\
    \bar v=0&\text{ on }\{0,\pi\}\times (0,\infty),\\
    \bar v\cdot \bar v_y =0 &\text{ on }[0,\pi]\times\{0\},\\
    \bar v\le C & \text{ in }[0,\pi]\times [0,\infty).
\end{cases}
\end{equation}
By \Cref{lem:uniqueness} below and the normalization we conclude that $\bar v= \sin{x} $.

For completeness, let us sketch a short proof of \eqref{eq:holderbound} for possibly unsigned solutions of $\lap \tilde v+ \tilde v=0,$ assuming $|\tilde v|\le1$ in $[0,\pi]\times[-2,2]$. This proof is inspired by a similar computation in \cite{FrS}.

Extend $\tilde v$ to be zero outside the strip and the set
\[
F\ :=\ [0,\sigma]\times\{0\}\ \cup\ \{0,\pi\}\times \R\quad \subset \quad \{v=0\}.
\]
If $B(x,r/2)$ does not intersect $F$, then by elliptic estimates
\[
\int_{B_{r/2}(x)}|\nabla \tilde v|^2 \le C r^2 .
\]
Due to the simple geometry of $F$, if $F\cap B(x,r/2)\neq \emptyset$, then Poincar\'e inequality applies in the annulus yielding
\begin{equation}\label{eq:poincare}
\int_{B_r(x)\setminus B_{r/2}(x)} \tilde v^2 \le C_F\, r^2 \int_{B_r(x)\setminus B_{r/2}(x)}|\nabla \tilde v|^2.
\end{equation}
We emphasize that $C_F$ can be taken {\it uniform in $\sigma$}.

Combining \eqref{eq:poincare} with the Caccioppoli inequality one finds
\[
\int_{B_{r/2}(x)}|\nabla \tilde v |^2 \le  C_F  \int_{B_r(x)\setminus B_{r/2}(x)}|\nabla \tilde v|^2 + Cr^2.
\]
Now the hole-filling trick and an iteration give, for some small $\alpha=\alpha(C_F)>0,$ the decay
\[
\int_{B_s(x)}|\nabla \tilde v|^2 \le C s^{2\alpha} \quad\text{for all}\quad x\in [0,\pi]\times[-1,1],\ s\in(0,\tfrac12],
\]
which is \eqref{eq:holderbound} by Campanato's criterion.
\end{proof}

\begin{lemma}\label{lem:uniqueness}
Assume that $w$ is a continuous function that vanishes on the sides of the infinite strip $[0,\pi]\times \R$, that solves
$$\lap w+w\ge0\quad \text{ in }\quad \{w>0\},$$ 
and that
$$0<\sup w\le 1.$$
Then $w$ is a multiple of $ \sin{x} $. 
\end{lemma}
\begin{proof}
By a standard barrier argument we can find a universal number $C$ such that
\[
 w\le C\sin (x)\quad \text{ for all }\quad x,y\in [0,\pi]\times \R.
\]
%
In other words, the nonnegative function
\[
A(y):=\sup_{x\in[0,\pi]}\frac{w^+(x,y)}{ \sin{x} },\quad y\in\R,
\]
is bounded above by a universal constant $C$.
By the maximum principle in $[0,\pi]\times[y_1,y_2]$ we find that \[
w(x,(1-t)y_1+ty_2)\le \big[(1-t)A(y_1)+tA(y_2)\big] \sin{x} , 
\]
for all $x\in[0,\pi]$ and $t\in[0,1]$, indeed functions of the form $(a+b y) \sin{x} $ solve the equation. 

This shows that the function $A$ is convex and bounded above, so $A(y)\equiv a>0.$

Now --- by the strong maximum principle --- if the weak supersolution $a  \sin{x} -w^+$ vanishes somewhere in the strip, then $w^+=w=a\, \sin{x} $ (here we use $a>0$). Otherwise, if
\[
a \sin{x} -w^+\ge \delta>0 \quad \text{in, say,}\quad [\tfrac\pi3,\tfrac{2\pi}{3}]\times [-3,3], 
\]
then --- using the Hopf Lemma barrier --- we have that 
\[
w^+-a \sin{x} \ge c(n)\delta \sin{x}  \quad \text{in}\quad \text [0,\pi]\times [-1,1], 
\]
this means that, by definition, $A(0)\ge a+c(n)\delta$ which is impossible since $A(0)=a$ as well.
\end{proof}

\section{A variant of the construction}\label{sec:variant}
In the construction that we presented in \Cref{sec:existence}, taking $v$ to be the {\it first} eigenfunction was convenient, but not necessary. We present a variation of the argument, this time taking the $k$-th eigenfunction with $k\lesssim m$. The key observation is the following result concerning nodal domains in thin tubes (the sets $\Omega_{m,\sigma}$ are the same as above). 
\begin{proposition}\label{prop:nodal}
Let $m\in\N$ be any integer and $\sigma\in[0,1]$. Assume $v$ solves $\lap_\S v+\mu(\mu+1) v=0$ in $\Omega_{m,\sigma}$, for some
\[m\le\mu\le (1+\gamma)\cdot m.\] 
Then the eigenvalue $\mu(\mu+1)$ is simple and $v$ does not change sign in $[0,\pi/m]\times [0,1/m]$.
The small constant $\gamma>0$ does not depend on $m$ nor $\sigma$.
\end{proposition}

This general principle is known, see for example the survey \cite{griesertubes} and the references therein. We give a self-contained proof at the end of this section. 

We show now that the construction carries out essentially as in the case $k=1$.

Fix any $m$ and $k$, with
$$m\text{ odd}\quad \text{and}\quad 1\le k\le 1 + \tfrac\gamma2 m.$$ 
For each $\sigma \in [0,1]$ pick the $k$-th eigenfunction $v_{k,\sigma}$, with homogeneity $\mu_k=\mu_k(\sigma)$.

For $\sigma=0$ we have
$$v_{k,0}=\rho^m \sin(m \theta) P_{2k-2}(\rho,z), \quad \quad \mu_k=m+ 2k-2,$$
with $P_{2(k-1)}$ a homogenous polynomial of degree $2(k-1)$ in $\rho$ and $z$, even in $z$.

For $\sigma=1$ we have
$$v_{k,1}=\rho^m \sin(m \theta) P_{2k-1}(\rho,z), \quad \quad \mu_k=m+ 2k-1,$$
with $P_{2k-1}$ a homogenous polynomial of degree $2k-1$ in $\rho$ and $z$, odd in $z$.

Thus, by the min-max formula:
\begin{equation}\label{eq:muk}
\begin{split}
\mu_k(\sigma)&\le \mu_k(1)= m+2k-1\le (1+\gamma) m,\\
\mu_k(\sigma)&\ge\mu_k(0)=m+2k-2,
\end{split}
\end{equation}
so we may apply \Cref{prop:nodal} and assume that $v_{k,\sigma} \ge 0$ on the equator and $v_{k,\sigma}$ is simple. Thus we can arrange $\sigma\mapsto v_{k,\sigma}$ to be continuous.

Construct $u_k$ from $v_k$ as before, so that 
$$ \partial_{\theta} u_k=v_k, \qquad  u_k(0,0)=0, \qquad \partial _\varphi u_k(0, {\tfrac{\pi}{m}})=0.$$
The $u_k$ is uniquely determined by these conditions, and we obtain as before
$$L_{\mu_k} u_k = c_k(\sigma) \delta_N,$$
with
\begin{equation}\label{eq:quantityk}
\frac{c_k(\sigma)}{2 m}=p_{\mu_k}(0) \int_0^{\pi/m} \p_\varphi v(s,0+) s\,ds + p_{\mu_k}'(0) \int_{0}^{\pi/m} v_k(s,0) (\tfrac{\pi}{m}-s )\, ds.
\end{equation}
Notice that, since $v_k\ge 0$ around the equator, $u_k$ satisfies the correct sign properties 
\begin{itemize}
    \item $u_k = 0$ and $\p_\varphi u_k\le0$ on the segment $[0,\sigma\pi/m]\times \{0\}$,
    \item $u_k \ge 0$ and $\p_\varphi u_k =0$ on the remaining part of the equator $[\tfrac{\sigma\pi}{m},\tfrac{\pi}{m}]\times\{0\},$
    \item $\p_\theta u_k  =0$ on the boundaries $\theta=0$ and $\theta=\pi/m$.
\end{itemize}
Now it is immediate to check that also in this case
\begin{equation}\label{eq:signsk}
c_k(0) \cdot c_k(1) <0, \end{equation}
which implies existence of a solution. Indeed, in the expression \eqref{eq:quantityk}, at the endpoints $\sigma=0,1$ the following happens:
\begin{itemize}
\item exactly one of the integrals involving $v$ vanishes (because $v$ is even or odd in $\varphi$) while the other integral is positive,
\item the coefficient of the integral that does not vanish flips sign, passing from $p_{\mu_k(0)}'(0)$ to $p_{\mu_k(1)}(0)$. When $m$ is odd these numbers have indeed opposite signs independently from $k$, by \Cref{prop:signs}. Indeed $\mu_k(0)$ has the same parity as $m$, while $\mu_k(1)$ has the same parity as $m+1$.
\end{itemize}

We are left with proving \Cref{prop:nodal}.

\begin{proof}[Proof of \Cref{prop:nodal}] 
It is enough to show that $v$ does not change sign in $[0,\pi/m]\times [0,1/m]$, since from this it follows that $\mu(\mu+1)$ is simple, by unique continuation.

We argue by compactness, letting in the rest of the proof
\begin{equation}\label{eq:11}
\gamma \to 0^+\quad\text{and}\quad m\to+\infty.
\end{equation}
We call universal constants that are uniform with respect to \eqref{eq:11} and $\sigma\in[0,1]$.

Consider the functions
\[
M^\pm(\varphi):=\max_{0\le \theta\le\pi/m} v^\pm(\theta,\varphi),\qquad \varphi\in (-\tfrac\pi2,\tfrac\pi2),
\]
where we extended $v$ evenly in the $\varphi$ variable.

\textbf{Claim 1.} 
Let $J\subset (-\tfrac\pi2,\tfrac\pi2)$ be an open connected component of $\{M^+>0\},$ and let $\varphi_0\in J$ be a maximum point for $M^+$ in $\overline J$. Then for $m$ large and $\gamma$ small,
\[
\cos\varphi_0\ge \tfrac18.\]
\textit{Proof of Claim 1.} 
For $d\in(0,1)$ to be fixed define the wedge
\[
W_{d}:=\{0<\theta<\tfrac\pi m\, , \rho<d\},
\]
recall that $\rho=\cos\varphi$. For $\alpha>0$ to be chosen define the barrier
\[
g_d:=  d^{-\beta}\rho^\beta\sin(m\theta),
\]
using $L_\mu(\rho^\beta)=\beta^2\rho^{\beta-2}+(\lambda_\mu-\lambda_\beta)\rho^\beta,$ one finds
\[
L_\mu g_d =d^{-\beta}\big\{\beta^2-m^2 +(\lambda_\mu-\lambda_\beta)\rho^2\big\}\rho^{\beta-2}\sin(m\theta).
\]
Thus, choosing $\beta:=m/2$ and any $d\le \tfrac14$, we have 
\begin{align*}
\beta^2+\lambda_\beta\rho^2 &\le \tfrac14 {m^2}\big(1+(1+o(1))d^2\big)\\
m^2-\lambda_\mu\rho^2 &\ge {m^2}\big(1-(1+o(1))d^2\big)
\end{align*}
where $o(1)$ is uniform in \eqref{eq:11}. 
This proves that
\[
g_d> 0\quad\text{and}\quad L_\mu g_d <0\quad\text{in}\quad W_{d},
\]
and that the maximum principle for $L_\mu$ holds in $W_{d}$.
The size of $g_d$ on different latitudes is
\[
\begin{cases}
g_d=3\sin(m\theta) &\text{ on }[0,\tfrac\pi m]\times \{\rho = d\},\\
g_d\le 2^{-m/2} &\text{ on }[0,\tfrac\pi m]\times \{\rho = \tfrac d2\}.
\end{cases}
\]
Let us assume by contradiction $\cos\varphi_0\le \tfrac 18$, so that we can choose $d$ as follows: 
\[
\rho_0:=\cos\varphi_0,\quad 2\rho_0=:\cos\varphi_1,\quad d:=2\rho_0\le \tfrac14.
\]
Along the edge $[0,\tfrac\pi m]\times\{\rho = \cos\varphi_1\}$ we have
\[
v\le C\, M^+(\varphi_1) \, g_{d},
\]
for some $C$ universal.
So, by the maximum principle in $W_d$, we have at $\rho=d$:
\[
M^+(\varphi_0)\le C\, 2^{-m/2} M^+(\varphi_1),
\]
which contradicts for $m$ large the maximality of $M^+(\varphi_0)$, unless $\varphi_1\notin\overline J$. But in this case a multiple of $g_d$ could touch from above $v^+\cdot\X_{J}(\varphi)$ in the interior of $W_d$, contradicting $L_\mu g_d<0$.

This concludes the proof of Claim 1.

Claim 1 ensures that the rescaled functions
\begin{equation}\label{eq:hatvdef}
\hat v(x,y):=v(x/m,\varphi_0 +y/m)
\end{equation}
solve a perturbation of $\lap \hat v +\hat v=0$, namely 
\[
\hat L \hat v := a_{ij} \hat v_{ij} + b_i \hat v_i + c \hat v =0,
\]
with 
\[
|a_{ij}-\delta_{ij}|+ |b_i|\le C_\ell/m^2\quad\text{and}\quad |c-1|\le C(\gamma +1/m),
\]
in $[0,\pi]\times(-\ell,\ell)$ for all $\ell>0$.

Let us denote with $\delta_0$ the number such that $\cos\varphi_0\ge\tfrac18$ if and only if $|\varphi|\le \tfrac\pi2-\delta_0$.

%

\textbf{Claim 2.} Let $\varphi_0\in(\tfrac2m,\tfrac\pi2-\delta_0)$ be any point such that $M^+(\varphi_0)>0.$ Then 
\begin{equation}\label{eq:ovidiusbound}
M^-(\varphi_0)\le C \|\hat v^+\|_{L^\infty([0,\pi]\times[-1,1])},
\end{equation}
where $\hat v$ is defined as in \eqref{eq:hatvdef}.

\textit{Proof of Claim 2.}  
Take the ball $B_r(z)$ centered at the point $z$ where the maximum of $v^-$ is realized on the segment $[0,\pi/m] \times \{\varphi_0\}$, so that $B_r(z)$ is tangent to the support of $v^+$ (which must intersect this segment as well). Notice that $r \le \pi/m$ by hypotheses and that we may assume $v^-(z)=1$. Then we need to show 
\begin{equation}\label{eq:w-}
\max_{B_{2r}(z)} v \ge  \delta_1,
\end{equation}
with $\delta_1>0$ small universal. 

First, by the boundary Harnack principle for $\hat L$, there is $\kappa >0$ small universal such that $B_{2\kappa r}(z)\subset \Omega$. So by the Harnack inequality
\[
v^-\ge c_0\quad\text{ in } \quad B_{\kappa r}(z)\subset \Omega,
\]
for some universal $c_0>0$.
We take the barrier $\psi\in C\big((\overline{B}_{2r}(z)\setminus B_{\kappa r }(z))\cap \overline{\Omega}\big)$ solving
\begin{equation*}
\begin{cases}
    L_\mu \psi=0& \text{ in }\big(B_{2r}(z)\setminus B_{\kappa r}\big)\cap\Omega,\\
     \psi =0& \text{ on }\big(B_{2r}(z)\setminus B_{\kappa r}\big)\cap\p\Omega,\\
    \psi =\delta_1&\text{ on }\p B_{2r}(z)\cap\Omega,\\
\psi =-c_0&\text{ on }\p B_{\kappa r}(z)\cap\Omega.
\end{cases}
\end{equation*}
By compactness and boundary gradient estimates, for $\delta_1>0$ universally small, we have that
\[
\psi<0\quad\text{ in } \quad \big( B_{3r/2}(z)\setminus B_{\kappa r}(z)\big)\cap \Omega.
\]
If \eqref{eq:w-} does not hold, then $v\le\psi$ since the maximum principle applies. But then this violates the definition of $r$, since we could have taken $3r/2$.
%
%
This concludes the proof of Claim 2.

We continue with the proof of \Cref{prop:nodal}. Notice that, by the maximum principle, if 
\begin{equation}\label{eq:lk}
M^+(\tfrac1m)\cdot M^-(\tfrac1m)=0,
\end{equation}
then 
\[
\text{$v$ cannot change sign in $[0,\tfrac\pi m]\times[-\tfrac1m,\tfrac1m]$},
\]
and we would conclude. 

So let us assume \eqref{eq:lk} does not hold, and take $J^+$ and $J^{-}$ to be the open connected components of $\{M^+>0\}$ and $\{M^-<0\}$ that contain $\tfrac1m$. Let $\varphi^+$ and $\varphi^-$ be the respective maximum points for $M^+$ and $M^-$ in these intervals. We may assume by symmetry
\[0\le\varphi^+\le \varphi^-.\] 
\textbf{Claim 3.} For $m$ large and $\gamma$ small,
\[
v\ge0 \quad\text{in}\quad [0,\pi/m]\times [\varphi^+-\tfrac2m,\varphi^++\tfrac2m].
\]
This leads to a contradiction since the interval $J^-$ must contain  the segment $[\tfrac1m,\varphi^-]$ by definition, thus also the segment $[\tfrac1m,\varphi^+]$, but we are claiming that cannot contain the segment $[\varphi^+-\tfrac2m,\varphi^+].$

\textit{Proof of Claim 3.} Consider the rescaled functions
\begin{align*}
w(x,y) &:= \frac{v(x/m, \varphi^++y/m)}{M^+(\varphi^+)} \cdot \X_{J^+}(\varphi^++y/m).
\end{align*}
Let $I\subset \R$ be the open interval such that
\[
m\cdot(J^+-\varphi^+)\to I,
\]
in the sense of pointwise convergence of indicator functions. We claim that necessarily $I=\R$. Indeed, if $I\subset (a,+\infty)$, a universal multiple of the function
\[
\sin x \cdot \cos\big(\frac\pi2 \frac{y-L}{a-L}\big)
\]
would lie above $w$ in $[0,\pi]\times[-a,L]$. For $L$ large enough in terms of $a$, this contradicts the fact that $w(\bar x,0)=1,$ for some $\bar x\in[0,\pi]$. A symmetric reasoning excludes $I\subset (-\infty,b).$

By Claim 2 and the $C^\alpha$ estimate \eqref{eq:holderbound}, we have
\[
w\to \bar w
\]
locally uniformly in $[0,\pi]\times \R$. This is clear in the case $m\varphi^+\to+\infty$, while, if $m\varphi^+ \to -y_1$ then one can use the maximum principle in $[0,\pi]\times[y_1-1,y_1+1],$ to prove that $|w|$ is universally bounded. 
Since $\lap \bar w +\bar w=0$ in the nonempty set $\{\bar w>0\}$, and $\sup\bar w= 1$ by definition, \Cref{lem:uniqueness} forces
\[
\bar w = \sin{x}.
\]

Now, if we had $\varphi^+\le3/m,$ $C^{1,\alpha}$ boundary estimates would give
\[
\nabla w \to \nabla \sin{x} \qquad \text{ uniformly in }[0,\pi]\times[4,5],
\]
forcing $w\ge \tfrac12 \sin{x} $ in $[0,\pi]\times\{\pm 4\}$ for $m$ and $1/\gamma$ large enough. Using that $v$ is even and the maximum principle, we would find that $w>0$ in $[0,\pi]\times[-4,+4]$, thus $M^-(1/m)=0,$ which contradicts our assumptions.

Similarly, if $\varphi^+\ge 3/m$, we use $C^{1,\alpha}$ boundary estimates in the box $[0,\pi]\times [-2,2]$. As before, these give $w\ge\tfrac12 \sin{x}\ge0 $ in that box for $m$ and $1/\gamma$ large enough, thus concluding the proof of Claim 3. 
\end{proof}
\begin{remark}
If for any $\sigma\in(0,1)$, any $m\in\N$ (possibly even) and any $k\in \N$, it happens that the $k$-th eigenvalue of $-\lap_\S$ in $\Omega_{m,\sigma}$ is {\it not } simple, then our method still gives a solution for the unsigned thin obstacle problem. Indeed, taking $v$ as an appropriate linear combination of two eigenfunctions, one can arrange $c_k(\sigma)=0$.
\end{remark}
\section{Legendre Functions}\label{sec:legendre}
We will need with some bounds on 
\[
L_\mu h =h'' -(n-2) \tan\varphi\cdot h' + \mu(\mu+n-2) h,
\]
seen as an ODE with initial conditions on $\varphi =0$, rather than $\varphi = \tfrac\pi2$ (compare \eqref{eq:ODE} with \eqref{2}).
\begin{lemma}\label{lem:logbound}
Let $h,g\colon[0,\pi/2)\to\R$ solve 
\begin{equation}\label{eq:ode}
	L_\mu h =h'' -(n-2) \tan\varphi\cdot h' + \mu(\mu+n-2) h = (n-2)\rho^{-1}g.
\end{equation}
Then 
\begin{equation}\label{eq:hblowupsize}
        |h'(\varphi)|\le C |\cos\varphi|^{2-n}\quad\text{for all}\quad\varphi\in[0,\pi/2),
\end{equation}
where the constant $C$ depends on upper bounds on the quantities 
\begin{equation}\label{eq:quantities}
|h(0)|,\quad |h'(0)|,\quad \sup|g|,\quad n\quad\text{and}\quad \mu.
\end{equation}
Furthermore, if $g\equiv 0$ and $h(\zeta)=0$ for some $\zeta\le \delta\le c$, then
\begin{equation}\label{eq:capbound}
    |h'(\varphi)-h'(\zeta) \cos(\sqrt{\lambda_\mu}\varphi)|\le C\delta^2|h'(\zeta)|\quad  \text{for all}\quad\varphi\in[\zeta,\delta].
\end{equation}
Here, $c$ and $C$ depend only on $n$, while $\delta$ is a free parameter.
\end{lemma}
Before embarking in the proof lets us comment this statement. \Cref{eq:hblowupsize} shows that our functions $u$, built in \Cref{subsec:vhu}, may in general blow-up at most logarithmically around the north pole.

The second bound, \eqref{eq:capbound}, shows that the $p_\mu$'s are $C^1$-approximated by planar waves of frequency $\sqrt{\lambda_\mu}$ and {\it constant} amplitude, in a strip  around the equator. As $\varphi\to\tfrac\pi2$, this behaviour is lost (cf. \Cref{fig:pmu}). 
This will be used in the proof of \Cref{prop:signs}.

\begin{proof}[Proof of \Cref{lem:logbound}]
Set $\omega^2:=\mu(\mu+n-2),$ and assume without loss of generality $n\ge3$. The equation reads as
\[
h''(\varphi)+\omega^2 h(\varphi) = (n-2)\tan(\varphi) h' +\cos(\varphi)^{-1}g(\varphi).
\]
Consider the Green function
\[
   G(\varphi):=\omega^{-1}\sin(\omega \varphi) \X_{(0,\infty)}(\varphi),
    \]
    which solves $G''+\omega^2 G = \delta_0$. Setting 
    $$f(\varphi):= (n-2)[\tan(\varphi) h'(\varphi) +\cos(\varphi)^{-1}g(\varphi)
]\X_{(0,\pi/2)}$$ 
we can represent $h$ as
    \begin{equation}\label{eq:rep}
    h(\varphi)= h(0)\cos(\omega \varphi ) +h'(0){\omega^{-1}\sin(\omega \varphi)}{} +(G*f)(\varphi).
    \end{equation}
    We use this formula to prove \eqref{eq:hblowupsize} denoting by $C$ a constant depending on the quantities \eqref{eq:quantities}. Differentiating we get
\begin{align*}
|h'(\varphi)|\le C + |(G'*f)(\varphi)|\le C(1+|\log\cos\varphi|)+ (n-2)\int_0^\varphi\tan(s) |h'(s)|\,ds,
\end{align*}
which for $F(\varphi):=\int_0^\varphi\tan(s) |h'(s)|\,ds$ becomes
\[
\big(\cos^{n-2}(\varphi)F(\varphi)\big)'\le C\cos^{n-1}(\varphi)(1+|\log\cos\varphi|)\in L^1(0,\tfrac\pi2),\]
so integrating $(\cos\varphi)^{n-2} F(\varphi)\le C$ and plugging back
\[
|h'(\varphi)|\le C(1+\log\cos\varphi)+F(\varphi)\le {C}{\cos(\varphi)^{2-n}}.
\]
This proves \eqref{eq:hblowupsize}, we turn to the proof of \eqref{eq:capbound}.

Differentiating the representation formula \eqref{eq:rep} and normalizing $h'(\zeta)=1$ we get
\[
\tilde h'(\varphi) = \cos(\omega\varphi) + (n-2)\int_0^\varphi \cos(\omega(\varphi-s))\tan(s+\zeta) \tilde h'(s)\, ds,
\]
for $\tilde h=h(\cdot +\zeta)$ and $\varphi\in[0,\delta-\zeta]$. This equation can be written as
\[
(id -\delta^2 T) \tilde h' = \cos(\omega\varphi),
\]
for a linear operator $T$ in $L^\infty([0,\delta-\zeta])$ with norm bounded by $n-2$. 
Now for $\delta$ dimensionally small $id-\delta^2T$ is invertible and \eqref{eq:capbound} follows by
\[
\|\tilde h'-\cos(\omega\varphi)\| =\|\big((id-\delta^2 T)^{-1} -id\big) \cos(\omega\varphi)\|\le C\delta^2 \|\cos(\omega\varphi)\|=C\delta^2.
\] 
\end{proof}

We turn to the proof of \Cref{prop:signs} whose statement we repeat for the reader's convenience.
\setcounter{lemma}{2} 
\begin{proposition}[Legendre functions]\label{prop:signsrep}\,
\begin{itemize}
\item[(a)] The signs of $p_\mu(0)$ and $p_\mu'(0)$ are given by $\cos(\mu \frac \pi 2)$ and $\sin(\mu \frac \pi 2),$ respectively.
\item[(b)] If $m$ is odd and $\mu\in(m,m+1],$ then \begin{equation}\label{eq:6}
-p_\mu'(0)/p_\mu(0) \ge c m (m+1-\mu).
\end{equation}
\item[(c)] If $m$ is odd and $\mu\in[m+1-\delta,m+1],$ then
\begin{equation}\label{eq:7}
0\le -p_\mu'(0)/p_\mu(0)\le Cm (m+1-\mu).
\end{equation}
\end{itemize}
The constants $c$, $C$ and $\delta>0$ depend only on $n$.
\end{proposition}

\begin{proof}[Proof of \Cref{prop:signsrep}]
We prove (a).

The form of the ODE implies that $p_\mu$ oscillates, and the zeros of $p_\mu$ and $p_\mu'$ are simple and intertwined. Indeed, near a critical point where $p_\mu$ is positive (negative) we know that $p_\mu$ is concave (convex), therefore $p_\mu$ must change sign between two consecutive critical points. 

Moreover, the number of zeros of $p_\mu$ in the interval $[0, \frac \pi 2]$ can only increase as we increase $\mu$. The reason is that between any two consecutive zeros of $p_\mu$ there must be at least a zero of $p_{\mu'}$ if $\mu' > \mu $. Otherwise in such and interval a positive multiple of $p_{\mu'}$ can be touched by below by a positive multiple of $p_\mu$ and we contradict the strong maximum principle at the contact point since $\lambda_{\mu'} > \lambda_\mu$. Finally, using the monotonicity of the first eigenvalue of a spherical cap, $p_{\mu'}$ must vanish somewhere after the first zero of $p_\mu$ (i.e., closer to $\tfrac\pi2$) .

Notice also that $p_\mu(0)=0$ precisely when $\mu$ is an odd integer, and $p_\mu'(0)=0$ precisely when $\mu$ is an even integer. This is because they correspond to the cases when $p_\mu$ is a polynomial.

Part (a) now follows from the continuity of the family $p_\mu$ with respect to the parameter $\mu$, as we increase $\mu$ from $0$ to infinity. When $\mu=0$ we have $p_0 \equiv 1$. The remarks above imply that the number of zeros and critical points of $p_\mu$ in the interval $[0,\pi/2]$ remains constant as $\mu$ ranges between two consecutive integers, and it increases exactly by one each time $\mu$ passes an integer. This means that at the initial point $\varphi=0$, $p_\mu$ and its derivative have the signs of $\sin(\mu \tfrac\pi2)$ and its derivative. 

We turn to the proof of (b) and (c). We denote by $c$ and $C$ constants that depend only on $n$.

Denote by $\zeta(\mu)>0$ the first (i.e., closest to 0) zero of $p_\mu,$ which depends smoothly on $\mu\in(m,m+2)$. Notice that, by part (a), $\zeta(m^+)=0^+$. 

Differentiating with respect to $\mu$ the equation \eqref{2} we find that
\[
\frac{\p}{\p\mu}p_\mu(\varphi)=:q_\mu(\varphi)
\]
solves
\begin{align}\label{4}
\begin{cases}
    L_\mu q_\mu = -(2\mu+n-2)  p_\mu\quad \text {in}\quad[0,\tfrac\pi2],\\
    q_\mu(\tfrac\pi2)=0, \quad \quad q_\mu'(\tfrac\pi2)=0.
\end{cases}
\end{align}
Thus differentiating $p_\mu(\zeta(\mu))=0$ with respect to $\mu$ we find
\begin{equation}\label{eq:zetaprime}
\zeta'(\mu) = -\frac{q_\mu(\zeta(\mu))}{p'_\mu(\zeta(\mu))}.
\end{equation}
We express the numerator integrating \eqref{4} on the spherical cap $\{\varphi\ge \zeta\},$
\begin{align*}
-p'_\mu(\zeta)q_\mu(\zeta)&=\int_{\{\varphi\ge\zeta\}} q_\mu L_\mu p_\mu -p_\mu L_\mu q_\mu =(2\mu+n-2)\int_{\{\varphi\ge\zeta\}}p_\mu^2\\
&=\frac{(2\mu+n-2)}{\mu(\mu+n-2)}\int_{\{\varphi\ge\zeta\}}|\nabla p_\mu|^2.
\end{align*}
Inserting this in \eqref{eq:zetaprime} we find
\[
\zeta'(\mu)= \frac{(2\mu+n-2)}{\mu(\mu+n-2)} \frac{\int_{\{\zeta\le\varphi\}}|\nabla p_\mu|^2}{(p'_\mu(\zeta))^2}.
\]
We claim that if $\zeta(\mu) \le c_0(n)$ then
\begin{equation}\label{eq:claim1}
c_1(n)\le \frac{\int_{\{\zeta\le\varphi\}}|\nabla p_\mu|^2}{(p'_\mu(\zeta))^2}\le C_1(n).
\end{equation}
This implies that $c_1\le m \zeta'\le C_1$ and that $\zeta(\mu)\le C_1/m$ for all $\mu\in (m,m+2)$.

Assuming the claim we can prove (b). 
Consider the rescaled functions
$$\tilde p_{\mu}(y):=\frac{p_{\mu}(y/m)}{p_{\mu}(0)}$$
which first vanish at $y_\mu:=m\zeta(\mu)\le C_1$. By smooth dependence from initial data of the ODE $\tilde L_\mu \tilde p_\mu =0$, and $p_{m+1}'(0)=0$, it follows that
\begin{align*}
|y_{m+1}-y_\mu|&\le C_2\big(|(\tilde p_{m+1}-\tilde p_\mu)(0)|+|(\tilde p_{m+1}'-\tilde p_\mu')(0)|\big)+\frac{C_2}{m}(m+1-\mu)\\
&=-C_2\tilde p_\mu'(0)+\frac{C_2}{m}(m+1-\mu),
\end{align*}
with $C_2>0$ uniform in $m$. We used that the derivative of the coefficients of $\tilde L_\mu$ with respect of $\mu$ is $O(1/m)$.
%
%
On the other hand, using \eqref{eq:zetaprime} we find
\[
y_{m+1}-y_\mu =  m\int_\mu^{m+1}\zeta'\ge c_1(m+1-\mu),
\]
thus (b) is proved for $m$ dimensionally large combining these bounds. For $m$ small, (b) follows by part (a) and smoothness.
%
%

The proof of (c) is similar. When $\mu$ is close to $m+1$, $\tilde p_\mu$ is close to $\cos y$, so we use the smooth dependence and the implicit function theorem to find 
\[
|y_{m+1}-y_\mu| \ge -c_2\tilde p_\mu'(0)\quad \text{for all}\quad\mu\in[m+1-\delta,m+1+\delta],
\]
for some $c_2>0$ uniform in $m$. On the other hand quation \eqref{eq:zetaprime} gives
\[
0\le y_{m+1}-y_\mu =  m\int_\mu^{m+1}\zeta'\le C_1(m+1-\mu),
\]
concluding the proof.
%
%

We are left with proving claim \eqref{eq:claim1}, we start with the upper bound.

Consider the energy density
\[
E(\varphi):= \tfrac12 (p_\mu'(\varphi))^2 + \tfrac\lambda 2 (p_\mu(\varphi))^2,
\]
which dissipates according to
\[
E'=E_\varphi = (n-2) \tan\varphi\,  (p_\mu')^2.
\]
We integrate this identity in $\{\varphi\ge \varphi_0\}$ where $\varphi_0$ is any zero of $p_\mu$, which will be chosen dimensionally small while large with respect to $\zeta(\mu)$. We find
\begin{align*}
(n-2)\int_{\{\varphi\ge \varphi_0\}}|\nabla p_\mu|^2 &= \int_{\{\varphi\ge \varphi_0\} }\cot\varphi\, E_\varphi \\
&=-\frac{(\cos\varphi_0)^{n-1}}{\sin\varphi_0}E(\varphi_0) + \int_{\{\varphi\ge \varphi_0\}}(n-1+z^{-2})E,
\end{align*}
we used that $p_\mu$ has finite energy to throw away the boundary term at the north pole in the integration by parts. Using the energy identity
\[
\int_{\{\varphi\ge \varphi_0\}}|\nabla p_\mu|^2 = \frac12\int_{\{\varphi\ge \varphi_0\}}|\nabla p_\mu|^2+\frac{\lambda}{2}\int_{\{\varphi\ge \varphi_0\}} p_\mu^2=\int_{\{\varphi\ge \varphi_0\}}E,\]
we arrive to
\[
{\sin\varphi_0}\int_{\{\varphi\ge \varphi_0\}}(1+z^{-2})E = {(\cos\varphi_0)^{n-1}}E(\varphi_0) = \tfrac12 {(\cos\varphi_0)^{n-1}}p_\mu'(\varphi_0)^2.
\]
This proves that for all $\varphi_0>0$ at which $p_\mu$ vanish we have
\begin{equation}\label{eq:l}
\int_{\{\varphi\ge \varphi_0\}}|\nabla p_\mu|^2 \le \frac{C}{{\varphi_0}}\,  p_\mu'(\varphi_0)^2.
\end{equation}
If we fix $\varphi_0(n)$ small, the uniform $C^1$ bound in \eqref{eq:capbound}, with $\delta=\varphi_0$, gives
\[
|p_\mu'(\varphi_0)|\le |p_\mu'(\varphi_0)-p_\mu'(\zeta)\cos(\sqrt\lambda\varphi_0)|+| p_\mu'(\zeta)\cos(\sqrt\lambda\varphi_0)|\le C|p_\mu'(\zeta)|,
\]
so \eqref{eq:l} becomes
\[
\int_{\{\varphi\ge \varphi_0\}}|\nabla p_\mu|^2 \le C(n)  p_\mu'(\zeta)^2.
\]
For the estimate in the region $\{\zeta\le \varphi\le \varphi_0\}$, we can use \eqref{eq:capbound} again, obtaining 
\[
\int_{\{\zeta\le \varphi\le \varphi_0\}}|\nabla p_\mu|^2 \le C  p_\mu'(\zeta)^2 \Big\{1+\int_{\zeta}^{\varphi_0} \cos(\sqrt\lambda \varphi)^2\,  d\varphi \Big\}\le C  p_\mu'(\zeta)^2.
\]
Thus the upper bound in \eqref{eq:claim1} is proved. 

The lower bound follows again just integrating $\eqref{eq:capbound}$ in the interval $[\zeta,\varphi_0]$
\[
\int_{\{\varphi\ge \zeta\}}|\nabla p_\mu|^2 \ge    p_\mu'(\zeta)^2 \int_{\zeta}^{\varphi_0} \cos(\sqrt\lambda \varphi)^2\,  d\varphi - C\varphi_0^3 p_\mu'(\zeta)^2\ge c(n) p_\mu'(\zeta)^2,
\]
since $\zeta\le C/m\le\varphi_0/4$ and $\varphi_0(n)$ can be chosen in such a way that
\[
\int_{\varphi_0/4}^{\varphi_0} \cos(\sqrt\lambda \varphi)^2\,  d\varphi-C\varphi_0^3\ge \varphi_0/4-C\varphi_0^3\ge \varphi_0/8.
\]
The point is that $\fint_{\varphi_0/4}^{\varphi_0}\cos(\sqrt\lambda \varphi)^2\,  d\varphi$ has size $\sim1$ independently from $\lambda$.
\end{proof}


\printbibliography
\end{document}